\documentclass [12pt]{amsart}
\usepackage{geometry}
\usepackage{graphicx}
\usepackage{amsthm}
\usepackage{amssymb}

\theoremstyle{plain}
\newtheorem{thm}{Theorem}[section]
\newtheorem{cor}[thm]{Corollary}
\newtheorem{lem}[thm]{Lemma}
\newtheorem{prop}[thm]{Proposition}
\newtheorem{defn}[thm]{Definition}
\newtheorem{exa}[thm]{Example}
\newtheorem{rem}[thm]{Remark}

\begin{document}

\title [{{ On Properties of Graded Rings and Graded Modules }}]{ On Properties of Graded Rings and Graded Modules }

\author[{{M. Refai }}]{\textit{Mashhoor Refai }}

\address
{\textit{Mashhoor Refai, President of Princess Sumaya University for Technology, Jordan.}}
\bigskip
{\email{\textit{m.refai@psut.edu.jo}}}

 \author[{{R. Abu-Dawwas }}]{\textit{Rashid Abu-Dawwas }}

\address
{\textit{Rashid Abu-Dawwas, Department of Mathematics, Yarmouk
University, Jordan.}}
\bigskip
{\email{\textit{rrashid@yu.edu.jo}}}


\date{}

\begin{abstract}  Let $R$ be a $G$-graded ring. In this article, we introduce two new concepts on graded rings, namely, weakly graded rings and invertible graded rings, and we discuss the relations between these concepts and several properties of graded rings. Also, we study the concept of weakly crossed products and study some properties defined on weakly crossed product and give the relationship between this new concept and several properties of graded rings. Moreover, in this article, we give a generalization for the concept of graded essential submodules, and introduce the concept of graded semi-uniform modules which is a generalization for the concept of graded uniform modules.
\end{abstract}

 \maketitle

 \section{Introduction}

 Let $G$ be a group with identity $e$ and $R$ be a ring with unity 1. Then $R$ is said to be $G$-graded ring if there exist additive subgroups $R_{g}$ of $R$ such that $R=\displaystyle\bigoplus_{g\in G}R_{g}$ and $R_{g}R_{h}\subseteq R_{gh}$ for all $g,h\in G$. A $G$-graded ring $R$ is denoted by $(R, G)$. The elements of $R_{g}$ are called homogeneous of degree $g$ and $R_{e}$ (the identity component of $R$) is a subring of $R$ and $1\in R_{e}$. For $x\in R$, $x$ can be written uniquely as $\displaystyle\sum_{g\in G}x_{g}$ where $x_{g}$ is the component of $x$ in $R_{g}$. The support of $(R, G)$ is defined by $supp(R,G)=\left\{g\in G:R_{g}\neq0\right\}$ and $h(R)=\displaystyle\bigcup_{g\in G}R_{g}$.

Let $R$ be a $G$-graded ring and $I$ an ideal of $R$. Then $I$ is said to be $G$-graded ideal if $I=\displaystyle\bigoplus_{g\in G}(I\cap R_{g})$, i.e., if $x\in I$ and $x=\displaystyle\sum_{g\in G}x_{g}$, then $x_{g}\in I$ for all $g\in G$. An ideal of a $G$-graded ring need not be $G$-graded, (see \cite{Dawwas}).

The concepts of faithful and non-degenerate graded rings have been introduced in \cite{Nastasescue}; $(R, G)$ is said to be faithful if for all $g, h\in G$, $a_{g}\in R_{g}-\{0\}$ we have $a_{g}R_{h}\neq\{0\}$ and $R_{h}a_{g}\neq\{0\}$. Clearly, if $(R, G)$ is faithful, then $supp(R, G)=G$. $(R, G)$ is said to be non-degenerate if for all $g\in G$, $a_{g}\in R_{g}-\{0\}$ we have $a_{g}R_{g^{-1}}\neq\{0\}$ and $R_{g^{-1}}a_{g}\neq\{0\}$. Otherwise, $(R, G)$ is said to be degenerate. Clearly, if $(R, G)$ is faithful, then $(R, G)$ is non-degenerate. However, the converse is not true in general and one can look in \cite{Nastasescue}.

Regular graded rings have been introduced in \cite{Nastasescue}; $(R, G)$ is said to be regular if for all $g\in G$, $a_{g}\in R_{g}-\{0\}$ we have $a_{g}\in a_{g}R_{g^{-1}}a_{g}$. It is easily to prove that if $(R, G)$ is regular, then $(R,G)$ is non-degenerate.

Strongly graded rings have been introduced in \cite{Nastasescue}; $(R,G)$ is said to be strong if $R_{g}R_{h}= R_{gh}$ for all $g,h\in G$. Clearly, if $(R,G)$ is strong, then $(R,G)$ is faithful. However, the converse is not true in general and one can look in \cite{Nastasescue}.

First strongly graded rings have been introduced and studied in \cite{Refai}; $(R,G)$ is said to be first strong if $1\in R_{g}R_{g^{-1}}$ for all $g\in supp(R,G)$. Clearly, if $(R,G)$ is strong, then $(R,G)$ is first strong. However, the converse is not true in general and one can look in \cite{Refai}.

Clearly, if $(R,G)$ is strong, then $supp(R,G)=G$. On the other hand, if $(R,G)$ is first strong, then $supp(R,G)$ is a subgroup of $G$. In fact, it has been proved that $(R,G)$ is first strong if and only if $supp(R,G)$ is a subgroup of $G$ and $R_{g}R_{h}= R_{gh}$ for all $g,h\in supp(R,G)$, (see \cite{Refai}).

Also, in \cite{Refai}, the concept of second strongly graded rings was introduced; $(R,G)$ is second strong if $supp(R,G)$ is a monoid in $G$ and $R_{g}R_{h}= R_{gh}$ for all $g,h\in supp(R,G)$. Clearly, if $(R,G)$ is strong, then $(R,G)$ is second strong. However, the converse is not true in general and one can look in \cite{Refai}. If $(R,G)$ is first strong, then $(R,G)$ is second strong. However, the converse is not true in general and one can look in \cite{Refai}.

In fact, if $(R,G)$ is second strong and $supp(R,G)$ is a subgroup of $G$, then $(R,G)$ is first strong. Moreover, $(R,G)$ is first strong if and only if $(R,G)$ is second strong and non-degenerate. Also, $(R,G)$ is strong if and only if $(R,G)$ is second strong and faithful, (see \cite{Refai}).

In this article, we introduce the concept of weakly graded rings; $(R,G)$ is said to be weak if whenever $g\in G$ with $R_{g}=\{0\}$, then $R_{g^{-1}}=\{0\}$. We discuss the relations between weakly graded rings and the concepts of strong, first strong, second strong, faithful, nondegenerate and regular graded rings, and then we give an analogous study to \cite{Dade}.

Also, we introduce the concept of invertible graded rings; $(R,G)$ is said to be invertible if the identity component $R_{e}$ is a field. We introduce an example on invertible graded rings which is not strong and an example on strongly graded rings which is not invertible. On the other hand, we prove that every invertible weakly graded domain is first strong. Several results are investigated, and then we study invertible graded rings as a vector space over $R_{e}$.

The concept of crossed product was introduced in \cite{Nastasescue}; $(R,G)$ is said to be crossed product if $R_{g}$ contains a unit for all $g\in G$. Clearly, if $(R,G)$ is a crossed product, then $supp(R,G)=G$. In this article, we study the concept of weakly crossed products (crossed products over the support) that was mentioned and used in \cite{Dawwas1}; $(R,G)$ is said to be weakly crossed product if $R_{g}$ contains a unit for all $g\in supp(R,G)$. We study some properties defined on weakly crossed product and give the relationship between this new concept and several properties of graded rings.

Let $M$ be a left $R$ - module. Then $M$ is a $G $-graded $R$-module if there exist additive subgroups $M_{g}$ of $M$ indexed by the elements $g\in G$ such that $M=\displaystyle\bigoplus_{g\in G}M_{g}$ and $R_{g}M_{h}\subseteq M_{gh}$ for all $g,h\in G$. The elements of $M_{g}$ are called homogeneous of degree $g$. If $x\in M$, then $x$ can be written uniquely as $\displaystyle\sum_{g\in G}x_{g}$, where $x_{g}$ is the component of $x$ in $M_{g}$. Clearly, $M_{g}$ is $R_{e}$-submodule of $M$ for all $g\in G$. Also, we write
$h(M)=\displaystyle\bigcup_{g\in G}M_{g}$ and $supp(M,G)=\left\{g\in G:M_{g}\neq0\right\}$.
Let $M$ be a $G$-graded $R$-module and $N$ be an $R$-submodule of $M$. Then $N$ is called $G$-graded $R$-submodule if $N=\displaystyle\bigoplus_{g\in
G}\left(N\bigcap M_{g}\right)$, i.e., if $x\in N$ and $x=\displaystyle\sum_{g\in G}x_{g}$, then $x_{g}\in N$ for all $g\in G$. Not all $R$-submodules of a $G$-graded $R$-module are $G$-graded, (see \cite{Dawwas} and \cite{Nastasescue}).

\begin{lem}(\cite{Farzalipour}) Let $R$ be a $G$-graded ring and $M$ be a $G$-graded $R$-module.

\begin{enumerate}

\item If $I$ and $J$ are graded ideals of $R$, then $I+J$ and $I\bigcap J$ are graded ideals of $R$.

\item If $N$ and $K$ are graded $R$-submodules of $M$, then $N+K$ and $N\bigcap K$ are graded $R$-submodules of $M$.

\item If $N$ is a graded $R$-submodule of $M$, $r\in h(R)$, $x\in h(M)$ and $I$ is a graded ideal of $R$, then $Rx$, $IN$ and $rN$ are graded $R$-submodules of $M$. Moreover, $(N:_{R}M)=\left\{r\in R:rM\subseteq N\right\}$ is a graded ideal of $R$.
\end{enumerate}
\end{lem}

In \cite{Farzalipour2}, it has been proved that if $N$ is a graded $R$-submodule of $M$, then $Ann(N)=\left\{r\in R:rN=0\right\}$ is a graded ideal of $R$.

A $G$-graded $R$-module $M$ is said to be strongly graded if $R_{g}M_{h}=M_{gh}$ for all $g,h\in G$. Clearly, $(R,G)$ is strong if and only if every graded $R$-module is strongly graded (see \cite{Nastasescue}). Also, $(M,G)$ is called first strong if $supp(R,G)$ is a subgroup of $G$ and $R_{g}M_{h}=M_{gh}$ for all $g\in supp(R,G)$, $h\in G$. Moreover, $(R,G)$ is first strong if and only if every graded $R$-module
is first strongly graded (see \cite{Refai Moh'd}).

Let $M$ be a $G$-graded $R$-module and $N$ be an $R$-submodule of $M$. Then $M/N$ may be made into a graded module by putting $(M/N)_{g}=(M_{g}+N)/N$ for all $g\in G$ (see \cite{Nastasescue}). In fact, the proof of the following will be a routine.

\begin{lem}\label{4} Let $M$ be a graded $R$-module, $K$ and $N$ be $R$-submodules of $M$ such that $K\subseteq N$. Then $N$ is a graded $R$-submodule of $M$ if and only if $N/K$ is a graded $R$-submodule of $M/K$.
 \end{lem}

 Graded essential submodules have been introduced and studied in \cite{Nastasescue}, a nonzero graded $R$-submodule $K$ of $M$ is said to be graded essential if $K\bigcap N\neq\{0\}$ for every nonzero graded $R$-submodule $N$ of $M$. Also, graded essential submodules have been studied in  \cite{Ceken}.

Graded prime ideals have been introduced and studied in \cite{Refai Hailat Obiedat}; a proper graded ideal $B$ of a graded ring $R$ is said to graded prime if whenever $a, b\in h(R)$ such that $ab\in B$, then either $a\in B$ or $b\in B$. Graded prime submodules have been introduced and studied in \cite{Atani}, a proper graded $R$-submodule $P$ of $M$ is said to be graded prime if whenever $r\in h(R)$ and $m\in h(M)$ such that $rm\in P$, then either $m\in P$ or $r\in (P:_{R}M)$. Graded prime submodules have been studied by several authors, for example \cite{Abu-Dawwas4} and \cite{Abu-Dawwas3}. It has been proved in \cite{Atani} that if $N$ is a graded prime $R$-submodule of $M$, then $(N:_{R}M)$ is a graded prime ideal of $R$.

 In this article, we introduce and study the concept of graded semi-essential submodules, a nonzero graded $R$-submodule $K$ of $M$ is said to be graded semi-essential if $K\bigcap P\neq\{0\}$ for every nonzero graded prime $R$-submodule $P$ of $M$. Clearly, every graded essential submodule is graded semi-essential, we prove that the converse is not true in general. We prove that the intersection of two graded semi-essential submodules is not graded semi-essential in general, however, it will be true if one is graded essential submodule. Also, we introduce another case where the intersection of two graded semi-essential submodules will be graded semi-essential.

Let $N$ be a graded $R$-submodule of $M$. Then the graded radical of $N$ is denoted by $Grad_{M}(N)$ and it is defined to be the intersection of all graded prime submodules of $M$ containing $N$. If there is no graded prime submodule containing $N$, then we take $Grad_{M}(N)=M$.

Let $M$ and $M^{\prime}$ be two $G$-graded $R$-modules. An $R$-homomorphism $f:M\rightarrow M^{\prime}$ is said to be graded $R$-homomorphism if $f(M_{g})\subseteq M^{\prime}_{g}$ for all $g\in G$. One can prove that if $K$ is a graded $R$-submodule of $M$ and $L$ is a graded $R$-submodule of $M^{\prime}$, then $f(K)$ is a graded $R$-submodule of $M^{\prime}$ and $f^{-1}(L)$ is a graded $R$-submodule of $M$ (see \cite{Nastasescue}). Let $f:M\rightarrow M^{\prime}$ be a graded $R$-epimorphism and $K$ be a graded $R$-submodule of $M$ such that $Ker(f)\subseteq K$. Then there exists a one to one order preserving correspondence between the proper graded $R$-submodules of $M$ containing $K$ and the proper graded $R$-submodules of $M^{\prime}$ containing $f(K)$. Moreover, for any graded $R$-submodule $L$ of $M^{\prime}$, there exists a graded $R$-submodule $N$ of $M$ such that $Ker(f)\subseteq N$ and $f(N)=L$.

Based on analogous results for graded rings, the proof of the following will be routine.

\begin{lem}\label{6} Let $f:M\rightarrow M^{\prime}$ be a graded $R$-epimorphism and $K$ be a graded $R$-submodule of $M$ such that $Ker(f)\subseteq K$.
\begin{enumerate}
\item If $N$ is a graded prime $R$-submodule of $M$ containing $K$, then $f(P)$ is a graded prime $R$-submodule of $M^{\prime}$ containing $f(K)$.
\item If $L$ is a graded prime $R$-submodule of $M^{\prime}$ containing $f(K)$, then $f^{-1}(L)$ is a graded prime $R$-submodule of $M$ containing $K$.
\end{enumerate}
\end{lem}

We study graded semi-essential submodules under homomorphisms.

In \cite{Escoriza}, a graded $R$-module $M$ is said to be graded multiplication if for every graded $R$-submodule $N$ of $M$, $N=IM$ for some graded deal $I$ of $R$. In this case, we can take $I=(N:_{R}M)$. Graded multiplication modules have been studied by several authors, for example, see \cite{Abu-Dawwa1}, \cite{Abu-Dawwas2} and \cite{Khaksari}. We prove that if $M$ is a graded multiplication faithful $R$-module and $K$ a graded $R$-submodule of $M$ such that $(K:_{R}M)$ is a graded semi-essential ideal of $R$, then $K$ is a graded semi-essential $R$-submodule of $M$.

The concept of graded uniform modules have been introduced and studied in \cite{Nastasescue}, a graded $R$ module $M$ is said to be graded uniform if every nonzero graded $R$-submodule of $M$ is graded essential. In this work, we introduce and study the concept of graded semi-uniform modules, a graded $R$ module $M$ is said to be graded semi-uniform if every nonzero graded $R$-submodule of $M$ is graded semi-essential. Clearly, every graded uniform module is graded semi-uniform, we prove that the converse is not true in general. It is known that the graded uniform property is hereditary, we prove that the graded semi-uniform property is not hereditary. We prove that if $R$ is a graded semi-uniform ring, then every graded multiplication faithful $R$-module is graded semi-uniform.

\section{Weakly Graded Rings}

In this section, we introduce the concept of weakly graded rings and discuss the relations between weakly graded rings and several properties of graded rings.

\begin{defn} Let $R$ be a $G$-graded ring. Then $(R,G)$ is said to be weak if whenever $g\in G$ with $R_{g}=\{0\}$, then $R_{g^{-1}}=\{0\}$.
\end{defn}

\begin{exa} The trivial graduation of a ring $R$ by a group $G$ is weak, that is $R_{e}=R$ and $R_{g}=\{0\}$ otherwise.
\end{exa}

\begin{exa}\label{3} Let $R=M_{2}(K)$ and $G=\mathbb{Z}_{4}$. Then $R$ is $G$-graded by

$R_{0}=\left(
         \begin{array}{cc}
           K & 0 \\
           0 & K \\
         \end{array}
       \right)$, $R_{2}=\left(
                          \begin{array}{cc}
                            0 & K \\
                            K & 0 \\
                          \end{array}
                        \right)$ and $R_{1}=R_{3}=\{0\}$. $(R,G)$ is first strong since $I\in R_{0}R_{0}$ and $I\in R_{2}R_{2}$ but $(R,G)$ is not strong since $R_{1}R_{3}=\{0\}\neq R_{0}$. Clearly, $(R, G)$ is weak.
\end{exa}

\begin{exa}\label{5} Let $R=K[x]$ (the ring of polynomials with coefficients from a field $K$) and $G=\mathbb{Z}$. Then $R$ is $G$-graded by $R_{j}=Kx^{j}$, $j\geq0$ and $R_{j}=\{0\}$ otherwise. $(R,G)$ is second strong but it is not first strong since $2\in supp(R,G)$ with $R_{2}R_{-2}=\{0\}\neq R_{0}$. $(R, G)$ is not weak since $R_{-1}=\{0\}$ but $R_{1}\neq\{0\}$.
\end{exa}

\begin{prop}\label{7} If $(R,G)$ is non-degenerate, then $(R,G)$ is weak.
\end{prop}

\begin{proof} Let $g\in G$ with $R_{g}=\{0\}$. If $R_{g^{-1}}\neq\{0\}$, then there exists $a_{g^{-1}}\in R_{g^{-1}}-\{0\}$ and since $(R,G)$ is non-degenerate, $a_{g^{-1}}R_{g}\neq\{0\}$ and then $R_{g}\neq\{0\}$ a contradiction. So, $R_{g^{-1}}=\{0\}$ and hence $(R,G)$ is weak.
\end{proof}

The next example shows that the converse of Proposition \ref{7} is not true in general.

\begin{exa} Let $R=M_{4}(K)$ and $G=D_{10}=\left\langle a,b:a^{5}=b^{2}=e,ba=a^{-1}b\right\rangle$. Then $R$ is $G$-graded by

$R_{e}=\left(
         \begin{array}{cccc}
           K & 0 & 0 & 0 \\
           0 & K & 0 & 0 \\
           0 & 0 & K & 0 \\
           0 & 0 & 0 & K \\
         \end{array}
       \right)$, $R_{a}=\left(
                          \begin{array}{cccc}
                            0 & K & 0 & 0 \\
                            0 & 0 & K & 0 \\
                            0 & 0 & 0 & 0 \\
                            0 & 0 & 0 & 0 \\
                          \end{array}
                        \right)$, $R_{a^{2}}=\left(
                                               \begin{array}{cccc}
                                                 0 & 0 & K & 0 \\
                                                 0 & 0 & 0 & 0 \\
                                                 0 & 0 & 0 & 0 \\
                                                 0 & 0 & 0 & 0 \\
                                               \end{array}
                                             \right)$, $R_{a^{3}}=\left(
                                                                    \begin{array}{cccc}
                                                                      0 & 0 & 0 & 0 \\
                                                                      0 & 0 & 0 & 0 \\
                                                                      K & 0 & 0 & 0 \\
                                                                      0 & 0 & 0 & 0 \\
                                                                    \end{array}
                                                                  \right)$, $R_{b}=\left(
                                                                                     \begin{array}{cccc}
                                                                                       0 & 0 & 0 & 0 \\
                                                                                       0 & 0 & 0 & K \\
                                                                                       0 & 0 & 0 & 0 \\
                                                                                       0 & K & 0 & 0 \\
                                                                                     \end{array}
                                                                                   \right)$, $R_{ab}=\left(
                                                                                                       \begin{array}{cccc}
                                                                                                         0 & 0 & 0 & K \\
                                                                                                         0 & 0 & 0 & 0 \\
                                                                                                         0 & 0 & 0 & 0 \\
                                                                                                         K & 0 & 0 & 0 \\
                                                                                                       \end{array}
                                                                                                     \right)$, $R_{a^{4}b}=\left(
                                                                                                                             \begin{array}{cccc}
                                                                                                                               0 & 0 & 0 & 0 \\
                                                                                                                               0 & 0 & 0 & 0 \\
                                                                                                                               0 & 0 & 0 & K \\
                                                                                                                               0 & 0 & K & 0 \\
                                                                                                                             \end{array}
                                                                                                                           \right)$, $R_{a^{4}}=\left(
                                                                                                                                                  \begin{array}{cccc}
                                                                                                                                                    0 & 0 & 0 & 0 \\
                                                                                                                                                     K& 0 & 0 & 0 \\
                                                                                                                                                     0& K & 0 & 0 \\
                                                                                                                                                     0& 0 & 0 & 0 \\
                                                                                                                                                  \end{array}
                                                                                                                                                \right)$ and
 $R_{a^{2}b}=R_{a^{3}b}=\{0\}$. $(R,G)$ is weak but it is degenerate since $x_{ab}=\left( \begin{array}{cccc}
                                                           0 & 0 & 0 & 0 \\
                                                           0 & 0 & 0 & 0 \\
                                                           0 & 0 & 0 & 0 \\
                                                           1 & 0 & 0 & 0 \\
                                                         \end{array}
                                                       \right)\in R_{ab}-\{0\}$ with $x_{ab}R_{(ab)^{-1}}=x_{ab}R_{a^{4}b}=\{0\}$.
\end{exa}

Since every faithful graded ring is non-degenerate and every regular graded ring is non-degenerate, we have the next two results.

\begin{cor}\label{11} If $(R,G)$ is faithful, then $(R,G)$ is weak.
\end{cor}

\begin{cor} If $(R,G)$ is regular, then $(R,G)$ is weak.
\end{cor}

\begin{prop} Let $R$ be a $G$-graded ring. If $R$ has no zero divisors, then $(R,G)$ is weak if and only if $(R,G)$ is non-degenerate.
\end{prop}

\begin{proof} Suppose that $(R,G)$ is weak. Let $g\in G$ and $a_{g}\in R_{g}-\{0\}$. Then $R_{g}\neq\{0\}$ and then $R_{g^{-1}}\neq\{0\}$ since $(R,G)$ is weak. Since $R$ has no zero divisors, $a_{g}R_{g^{-1}}\neq\{0\}$ and $R_{g^{-1}}a_{g}\neq\{0\}$ and hence $(R,G)$ is non-degenerate. The converse holds by Proposition \ref{7}.
\end{proof}

\begin{prop}\label{15} Let $R$ be a weakly $G$-graded ring. If $I$ is an ideal of $R$ with $I\bigcap R_{e}=\{0\}$ and $R$ has no zero divisors, then $I\bigcap R_{g}=\{0\}$ for all $g\in G$.
\end{prop}

\begin{proof} Let $g\in G$. If $g\notin supp(R,G)$, then $R_{g}=\{0\}$ and then $I\bigcap R_{g}=\{0\}$. Suppose that $g\in supp(R,G)$. Then $\{0\}=I\bigcap R_{e}\supseteq I\bigcap (R_{g}R_{g^{-1}})\supseteq(I\bigcap R_{g})R_{g^{-1}}$. Since $R$ has no zero divisors, either $I\bigcap R_{g}=\{0\}$ or $R_{g^{-1}}=\{0\}$. If $R_{g^{-1}}=\{0\}$, then $R_{g}=\{0\}$ since $(R,G)$ is weak and that is a contradiction since $g\in supp(R,G)$. So, $I\bigcap R_{g}=\{0\}$.
\end{proof}

\begin{prop}\label{8} Let $R$ be a $G$-graded ring. If $supp(R,G)$ is a subgroup of $G$, then $R$ is weakly graded.
\end{prop}

\begin{proof} Let $g\in G$ with $R_{g}=\{0\}$. If $R_{g^{-1}}\neq\{0\}$, then $g^{-1}\in supp(R,G)$ and since $supp(R,G)$ is a subgroup, $g\in supp(R,G)$ and then $R_{g}\neq\{0\}$ a contradiction. So, $R_{g^{-1}}=\{0\}$ and hence $R$ is weakly graded.
\end{proof}

The next example shows that the converse of Proposition \ref{8} is not true in general.

\begin{exa}\label{13} Let $R=M_{3}(K)$ and $G=\mathbb{Z}_{7}$. Then $R$ is $G$-graded by

$R_{0}=\left(
         \begin{array}{ccc}
           K & 0 & 0 \\
           0 & K & 0 \\
           0 & 0 & K \\
         \end{array}
       \right)$, $R_{1}=\left(
                          \begin{array}{ccc}
                            0 & K & 0 \\
                            0 & 0 & K \\
                            0 & 0 & 0 \\
                          \end{array}
                        \right)$, $R_{2}=\left(
                                           \begin{array}{ccc}
                                             0 & 0 & K \\
                                             0 & 0 & 0 \\
                                             0 & 0 & 0 \\
                                           \end{array}
                                         \right)$, $R_{5}=\left(
                                                            \begin{array}{ccc}
                                                              0 & 0 & 0 \\
                                                              0 & 0 & 0 \\
                                                              K & 0 & 0 \\
                                                            \end{array}
                                                          \right)$, $R_{6}=\left(
                                                                             \begin{array}{ccc}
                                                                               0 & 0 & 0 \\
                                                                               K & 0 & 0 \\
                                                                               0 & K & 0 \\
                                                                             \end{array}
                                                                           \right)$ and $R_{3}=R_{4}=\{0\}$. Clearly, $(R,G)$ is weak but $supp(R,G)=\left\{0,1,2,5,6\right\}$ is not a subgroup of $G$.
\end{exa}

The next result states that the converse of Proposition \ref{8} is true if $R$ has no zero divisors.

\begin{prop}\label{9} Let $R$ be a $G$-graded ring. If $R$ has no zero divisors, then $(R,G)$ is weak if and only if $supp(R,G)$ is a subgroup of $G$.
\end{prop}

\begin{proof} Suppose that $(R,G)$ is weak. Let $g,h\in supp(R,G)$. Then $\{0\}\neq R_{g}R_{h}\subseteq R_{gh}$ and then $gh\in supp(R,G)$. Since $0\neq1\in R_{e}$, $e\in supp(R,G)$. Let $g\in supp(R,G)$. Then $g^{-1}\in supp(R,G)$ since $(R,G)$ is weak. So, $supp(R,G)$ is a subgroup of $G$. The converse holds by Proposition \ref{8}.
\end{proof}

\begin{cor}\label{10} Let $R$ be an integral domain. If $R$ is weakly $G$-graded, then $supp(R,G)$ is an abelian subgroup of $G$.
\end{cor}

\begin{proof} By Proposition \ref{9}, $supp(R,G)$ is a subgroup of $G$. Let $g,h\in supp(R,G)$. Then $\{0\}\neq R_{g}R_{h}\subseteq R_{gh}$ and $\{0\}\neq R_{h}R_{g}\subseteq R_{hg}$. Since $R$ is commutative, $R_{g}R_{h}=R_{h}R_{g}$ and then $\{0\}\neq R_{g}R_{h}\subseteq R_{gh}\bigcap R_{hg}$ and hence $gh=hg$. So, $supp(R,G)$ is an abelian subgroup of $G$.
\end{proof}

\begin{cor} Let $R$ be an integral domain. If $R$ is $G$-graded such that $(R,G)$ is faithful, then $G$ is abelian.
\end{cor}

\begin{proof} By Corollary \ref{11}, $(R,G)$ is weak and then by Corollary \ref{10}, $supp(R,G)$ is an abelian subgroup of $G$. On the other hand, since $(R,G)$ is faithful, $supp(R,G)=G$ and hence $G$ is abelian.
\end{proof}

In \cite{Refai}, it is proved that if $(R,G)$ is first strong, then $supp(R,G)$ is a subgroup of $G$. Combining this with Proposition \ref{8} to have the next result.

\begin{cor}\label{12} If $R,G)$ is first strong, then $(R,G)$ is weak.
\end{cor}

The converse of Corollary \ref{12} is not true in general; in Example \ref{13}, $(R,G)$ is weak but it is not first strong since $supp(R,G)$ is not a subgroup of $G$. Also, it is clear that if $(R,G)$ is strong, then $supp(R,G)=G$. Combining this with Proposition \ref{8} to have the next result.

\begin{cor}\label{14} If $R,G)$ is strong, then $(R,G)$ is weak.
\end{cor}

The converse of Corollary \ref{14} is not true in general; in Example \ref{13}, $(R,G)$ is weak but it is not strong since Since $R_{3}R_{4}=\{0\}\neq R_{0}$.

On the other hand, if $(R,G)$ is second strong, then $(R,G)$ need not be weak; in Example \ref{5}, $(R,G)$ is second strong but it is not weak. Also, the next example shows that if $(R,G)$ is weak, then $(R,G)$ need not be second strong.

\begin{exa} Let $R=K[x]$ and $G=\mathbb{Z}_{3}$. Then $R$ is $G$-graded by

$R_{0}=\left\langle 1,x^{3},x^{6},....\right\rangle$, $R_{1}=\left\langle x,x^{4},x^{7},....\right\rangle$ and $R_{2}=\left\langle x^{2},x^{5},x^{8},....\right\rangle$. $(R,G)$ is weak but it is not second strong since $R_{1}R_{2}\neq R_{0}$ since $1\in R_{0}$ such that $1\notin R_{1}R_{2}$.
\end{exa}

In fact, if $(R,G)$ is second strong and weak, then $(R,G)$ is first strong and hence $supp(R,G)$ is a subgroup of $G$. This is what we are going to prove.

\begin{prop} Let $R$ be a $G$-graded ring. If $(R,G)$ is second strong, then $(R,G)$ is weak if and only if $(R,G)$ is first strong.
\end{prop}

\begin{proof} Suppose that $(R,G)$ is weak. Let $g\in supp(R,G)$. Then since $(R,G)$ is weak, $g^{-1}\in supp(R,G)$ and then since $(R,G)$ is second strong, $R_{g}R_{g^{-1}}=R_{e}$ and then $1\in R_{g}R_{g^{-1}}$ and hence $(R,G)$ is first strong. The converse holds by Corollary \ref{12}.
\end{proof}

\begin{cor} Let $R$ be a $G$-graded ring. If $(R,G)$ is second strong, then $(R,G)$ is weak if and only if $supp(R,G)$ is a subgroup of $G$.
\end{cor}

In the rest of this section, we introduce an analogous study to that given in \cite{Dade}. The question we deal with is: For a group $G$ and a field $K$, can we write $R=K[x_{1},x_{2},....,x_{m}]$ as strongly $G$-graded ring?.

The next example shows that it is possible for $G=\mathbb{Z}_{2}$ and $K=\mathbb{C}$ (the field of complex numbers).

\begin{exa} Let $R=\mathbb{C}[x_{1},x_{2},....x_{m}]$ and $G=\mathbb{Z}_{2}$. Then $R$ is strongly $G$-graded by $R_{0}=\mathbb{R}[x_{1},x_{2},....x_{m}]$ and $R_{1}=i\mathbb{R}[x_{1},x_{2},....x_{m}]$ where $\mathbb{R}$ is he field of real numbers.
\end{exa}

\begin{lem}\label{16} Let $R$ be a weakly $G$-graded ring ($R$ is any ring). If $R$ has no zero divisors and $R_{e}$ is a division ring, then $R$ is simple.
\end{lem}

\begin{proof} Let $I$ be a nonzero ideal of $R$. If $I\bigcap R_{e}=\{0\}$, then by Proposition \ref{15}, $I\bigcap R_{g}=\{0\}$ for all $g\in G$ and then $I=\{0\}$ a contradiction. So, $I\bigcap R_{e}$ is a nonzero ideal of $R_{e}$ and then since $R_{e}$ is division ring, $I\bigcap R_{e}=R_{e}$. Let $g\in G$. Then $R_{g}\subseteq R_{g}R_{e}=R_{g}(I\bigcap R_{e})\subseteq R_{g}I\bigcap R_{g}R_{e}\subseteq R_{g}I\bigcap R_{g}\subseteq I\bigcap R_{g}$ and hence $I\bigcap R_{g}=R_{g}$, i.e., $R_{g}\subseteq I$ for all $g\in G$. Thus, $I=R$ and hence $R$ is simple.
\end{proof}

By Lemma \ref{16}, we can introduce the following:

\begin{lem} Let $R$ be a weakly $G$-graded ring ($R$ is any ring). If $R$ has no zero divisors and $R_{e}$ is a division ring, then $R$ is Artinian (Noetherian).
\end{lem}

\begin{prop}\label{17} Let $R=K[x_{1},x_{2},....x_{m}]$ be a $G$-graded. If $R_{e}\subseteq K$, then $(R,G)$ is not weak.
\end{prop}

\begin{proof} Firstly, we prove that $R_{e}$ is a subfield of $K$. Clearly, $R_{e}$ is a commutative ring with unity. Let $k\in R_{e}-\{0\}$ and $k^{-1}=r_{g_{1}}+r_{g_{2}}+....+r_{g_{n}}$ where $r_{g_{i}}\in R_{g_{i}}-\{0\}$. Then $1=kk^{-1}=kr_{g_{1}}+kr_{g_{2}}+....+kr_{g_{n}}$ with $kr_{g_{i}}\in R_{g_{i}}$ and $1\in R_{e}$. So, $n=1$ and $g_{1}=e$. Hence, $k^{-1}\in R_{e}$. Let $f$ be a nonconstant homogeneous element of $R$. Then $I_{a}=\left\langle rf^{n+a}:r\in R,n\in \mathbb{N}\right\rangle$ ($a\in \mathbb{N}$) is an ideal of $R$ with $I_{1}\supset I_{2}\supset....$ but $f^{n}\notin I_{n}$ for all $n\in \mathbb{N}$; let $a\in \mathbb{N}$ and $g\in I_{a+1}$. Then $g=rf^{n+(a+1)}$ for some $r\in R$ and $n\in \mathbb{N}$ and then $g=rf^{(n+1)+a}\in I_{a}$. Suppose $f^{n}\in I_{n}$ for some $n\in \mathbb{N}$. Then $f^{n}=rf^{t+n}$ for some $t\in \mathbb{N}$ and then since $R$ is integral domain, $1=rf^{t}$ a contradiction since $t\in \mathbb{N}$ and $f$ is nonconstant. Therefore, there is no $s\in \mathbb{N}$ with $I_{n}=I_{s}$ for all $n\geq s$, i.e., $R$ is not Artinian. Hence, by Lemma \ref{16}, $(R,G)$ is not weak.
\end{proof}

\begin{cor} Let $R=K[x_{1},x_{2},....x_{m}]$ be a $G$-graded. If $R_{e}\subseteq K$, then $(R,G)$ is not strong.
\end{cor}

\begin{proof} If $(R,G)$ is strong, then by Corollary \ref{14}, $(R,G)$ is weak and this contradicts Proposition \ref{17}. Hence, $(R,G)$ is not strong.
\end{proof}

\section{Invertible Graded Rings}

In this section, we introduce the concept of invertible graded rings and we study this concept assuming that $R$ is a weakly graded domain.

\begin{defn} Suppose that $R$ is a $G$-graded ring. Then $(R,G)$ is said to be invertible if the identity component $R_{e}$ is a field.
\end{defn}

\begin{exa}\label{2.2} Consider $R=\mathbb{C}$ (the field of complex numbers) and $G=\mathbb{Z}_{2}$. Then $R$ is $G$-graded by $R_{0}=\mathbb{R}$ and $R_{1}=i\mathbb{R}$. Since $R_{0}$ is a field, $(R,G)$ is invertible.
\end{exa}

The next example shows that if $(R,G)$ is invertible, then $R$ need not to be a field in general.

\begin{exa}\label{2.3} Consider $R=K[x]$ (where $K$ is a field) and $G=\mathbb{Z}$. Then $R$ is $G$-graded by $R_{j}=Kx^{j}$ for $j\geq0$ and $R_{j}=\{0\}$ otherwise. Since $R_{0}=K$ is a field, $(R,G)$ is invertible. However, $R$ it self is not a field.
\end{exa}

\begin{lem}\label{2.4} Suppose that $R$ is a $G$-graded ring. If $g\in G$ and $r\in R_{g}$ is an unit, then $r^{-1}\in R_{g^{-1}}$.
\end{lem}

\begin{proof} By (\cite{Nastasescue}, Proposition 1.1.1), $r^{-1}\in R_{h}$ for some $h\in G$ and then $rr^{-1}\in R_{g}R_{h}\subseteq R_{gh}$. On the other hand, $rr^{-1}=1\in R_{e}$. So, $0\neq rr^{-1}\in R_{e}\bigcap R_{gh}$ and then $gh=e$, i.e., $h=g^{-1}$. Hence, $r^{-1}\in R_{g^{-1}}$.
\end{proof}

\begin{lem}\label{2.5} Let $R$ be a weakly $G$-graded domain. If $(R,G)$ is invertible, then every nonzero homogeneous element is unit.
\end{lem}

\begin{proof} Let $g\in G$ and $r\in R_{g}-\{0\}$. Then since $(R, G)$ is weak, $R_{g^{-1}}\neq\{0\}$ and then since $R$ is a domain, $I=R_{g^{-1}}r$ is a nonzero ideal of $R_{e}$ and since $R_{e}$ is a field, $I=R_{e}$ and then $1\in I$. It follows that $1=xr$ for some $x\in R_{g^{-1}}$ and hence $r$ is unit.
\end{proof}

\begin{prop}\label{2.6} Let $R$ be a weakly $G$-graded domain. If $(R,G)$ is invertible, then $R_{g}$ is cyclic $R_{e}$-module for all $g\in G$.
\end{prop}

\begin{proof} Let $g\in G$. If $R_{g}=\{0\}$, then it is done. Suppose that $R_{g}\neq\{0\}$. Then there exists a nonzero $r\in R_{g}$ and then by Lemma \ref{2.4} and Lemma \ref{2.5}, $r$ is unit with $r^{-1}\in R_{g^{-1}}$. Let $x\in R_{g}$. Then $x=x.1=xr^{-1}r\in R_{g}R_{g^{-1}}r\subseteq R_{e}r$ and then $R_{g}\subseteq R_{e}r\subseteq R_{e}R_{g}\subseteq R_{g}$. Hence, $R_{g}=R_{e}r$, i.e., $R_{g}$ is cyclic $R_{e}$-module.
\end{proof}

\begin{cor}\label{2.7} Let $R$ be a weakly $G$-graded domain. If $(R,G)$ is invertible and $supp(R,G)$ is finite, then $R$ is Noetherian.
\end{cor}

\begin{proof} Since supp$(R,G)$ is finite, $R=\bigoplus_{i=1}^{n}R_{g_{i}}$. By Proposition \ref{2.6}, $R_{g_{i}}$ is cyclic $R_{e}$-module for all $i\leq i\leq n$. On the other hand, since $R_{e}$ is a field, $R_{e}$ is Noetherian. So, each $R_{g_{i}}$ is Noetherian and hence $R$ is Noetherian.
\end{proof}

\begin{prop}\label{2.8} Let $R$ be a weakly $G$-graded domain. If $(R,G)$ is invertible, then $R_{g}$ is simple $R_{e}$ module for all $g\in G$.
\end{prop}

\begin{proof} Let $g\in G$. If $R_{g}=\{0\}$, then it is done. Suppose that $R_{g}\neq\{0\}$. Then by Proposition \ref{2.6}, $R_{g}=R_{e}r$ for some $r\in R_{g}$. Let $N$ be a nonzero $R_{e}$-submodule of $R_{g}$ and let $I=\{x\in R_{e}:xr\in N\}$. Then clearly, $I$ is an ideal of $R_{e}$. We show that $N=Ir$. Let $n\in N\subseteq R_{g}$. Then $n=xr$ for some $x\in R_{e}$ and then $x\in I$ and hence $n\in Ir$. Thus, $N\subseteq Ir$. Let $n\in Ir$. Then $n=xr$ for some $x\in I$. Since $x\in I$, $xr\in N$ and then $n\in N$. Thus, $Ir\subseteq N$. Therefore, $N=Ir$. Now, if $I=\{0\}$, then $N=\{0\}$ a contradiction. So, $I$ is a nonzero ideal of $R_{e}$ and since $R_{e}$ is a field, $I=R_{e}$ and then $N=R_{e}r=R_{g}$. Hence, $R_{g}$ is simple $R_{e}$-module.
\end{proof}

If $(R,G)$ is invertible, then $R$ need not to be strongly $G$-graded. To see this, in Example \ref{2.3}, $(R,G)$ is invertible but $R$ is not strongly $G$-graded since $1\notin R_{1}R_{-1}=\{0\}$. Also, if $R$ is strongly $G$-graded, then $(R,G)$ need not to be invertible, see the following example.

\begin{exa}\label{2.9} Consider $R=\mathbb{Z}[i]$ and $G=\mathbb{Z}_{2}$. Then $R$ is $G$-graded by $R_{0}=\mathbb{Z}$ and $R_{1}=i\mathbb{Z}$. Since $1\in R_{0}R_{0}$ and $1\in R_{1}R_{1}$, $R$ is strongly $G$-graded. However, $(R,G)$ is not invertible since $R_{0}=\mathbb{Z}$ is not a field.
\end{exa}

In fact, we prove that every invertible weakly graded domain is first strongly graded.

\begin{prop}\label{2.10} Let $R$ be a weakly $G$-graded domain. If $(R,G)$ is invertible, then $R$ is first strongly $G$-graded.
\end{prop}

\begin{proof} Let $g\in supp(R, G)$. Then there exists a nonzero $r\in R_{g}$ and then by Lemma \ref{2.4} and Lemma \ref{2.5}, $r$ is unit with $r^{-1}\in R_{g^{-1}}$. So, $1=rr^{-1}\in R_{g}R_{g^{-1}}$ and hence $R$ is first strongly $G$-graded.
\end{proof}

 \begin{defn}(\cite{Nastasescue}) A graded ring $R$ is said to be graded simple if the only graded ideals of $R$ are $\{0\}$ and $R$ itself.
 \end{defn}

\begin{prop}\label{2.12} Let $R$ be a weakly $G$-graded domain. If $(R,G)$ is invertible, then $R$ is graded simple.
\end{prop}

\begin{proof} Let $I$ be a nonzero graded ideal of $R$. Then there exists $g\in G$ such that $I_{g}\neq\{0\}$, i.e., $I_{g}$ is a nonzero $R_{e}$-submodule of $R_{g}$. By Proposition \ref{2.8}, $R_{g}$ is simple $R_{e}$-module and then $I_{g}=R_{g}$. On the other hand, $I_{g}=I\bigcap R_{g}$ and then $R_{g}\subseteq I$. Now, by Proposition \ref{2.10}, $R$ is first strongly $G$-graded and then $1\in R_{e}=R_{g^{-1}}R_{g}\subseteq R_{g^{-1}}I\subseteq I$ and hence $I=R$. Therefore, $R$ is graded simple.
\end{proof}

In the rest of this section, we suppose that $R$ is a $G$-graded ring such that $(R,G)$ is invertible. Then we can consider $R$ as a vector space over $R_{e}$, i.e., $R_{e}$ is the field of scalars.

\begin{prop}\label{3.1} Suppose that $R$ is a non-trivially $G$-graded ring such that $(R,G)$ is invertible. Define $T:R\rightarrow R_{e}$ by $T(x)=x_{e}$. Then

\begin{enumerate}

\item $T$ is Linear transformation.

\item $KerT\bigcap R_{e}=\{0\}$.

\item $R_{g}\subseteq Ker(T)$ for all $g\in G-\{e\}$.

\item $T$ is not injective.

\item $T$ is surjective.
\end{enumerate}
\end{prop}

\begin{proof}

\begin{enumerate}

\item Let $x,y\in R$ and $r\in R_{e}$. Then $T(x+y)=(x+y)_{e}=x_{e}+y_{e}=T(x)+T(y)$ and $T(rx)=(rx)_{e}=rx_{e}=rT(x)$ (since $r\in R_{e}$). Hence, $T$ is linear.

\item Let $x\in Ker(T)\bigcap R_{e}$. Then $x\in R_{e}$ with $T(x)=0$ and then $0=T(x)=x_{e}=x$ (since $x\in R_{e}$). Hence, $Ker(T)\bigcap R_{e}=\{0\}$.

\item Let $g\in G-\{e\}$ and $x\in R_{g}$. Then $T(x)=x_{e}=0$ (since $x\in R_{g}$ and $g\neq e$), i.e., $x\in Ker(T)$. Hence, $R_{g}\subseteq Ker(T)$.

\item Since $R$ is non-trivially graded, there exists $g\in G-\{e\}$ such that $R_{g}\neq \{0\}$ and then by (3), $\{0\}\neq R_{g}\subseteq Ker(T)$, i.e., $Ker(T)\neq \{0\}$ and hence $T$ is not injective.

\item Let $x\in R_{e}$. Then $x\in R$ such that $T(x)=x_{e}=x$ (since $x\in R_{e}$). Hence, $T$ is surjective.
\end{enumerate}
\end{proof}

\begin{lem}\label{3.2} If $(R,G)$ is invertible, then $R_{g}$ is a subspace of $R$ for all $g\in G$.
\end{lem}

\begin{proof} Let $g\in G$. Since $R_{g}$ is additive subgroup, $x+y\in R_{g}$ for all $x,y\in R_{g}$. Let $r\in R_{e}$ and $x\in R_{g}$. Then $rx\in R_{e}R_{g}\subseteq R_{g}$. Hence, $R_{g}$ is a subspace of $R$.
\end{proof}

\begin{prop}\label{3.3} If $(R,G)$ is invertible, then $(R_{g}+R_{h})/R_{g}\approx R_{h}$ for all $g,h\in G$, $g\neq h$.
\end{prop}

\begin{proof} Let $g,h\in G$ such that $g\neq h$. Clearly, $R_{g}$ is a subspace of $R_{g}+R_{h}$. Define $T:R_{h}\rightarrow (R_{g}+R_{h})/R_{g}$ by $T(x)=x+R_{g}$. Then it is easy to prove that $T$ is well-defined linear transformation. To prove that $T$ is surjective, let $y\in (R_{g}+R_{h})/R_{g}$. Then $y=r+R_{g}$ for some $r\in R_{g}+R_{h}$ and then $y=(a+b)+R_{g}$ for some $a\in R_{g}$ and $b\in R_{h}$ and so $y=b+R_{g}$ and hence $T(b)=y$. Therefore, $T$ is surjective. So, $R_{h}/Ker(T)\approx (R_{g}+R_{h})/R_{g}$. On the other hand, $Ker(T)=\{x\in R_{h}:x+R_{g}=\{0\}\}=\{x\in R_{h}:x\in R_{g}\}=R_{h}\bigcap R_{g}=\{0\}$ (since $g\neq h$). Thus, $(R_{g}+R_{h})/R_{g}\approx R_{h}/\{0\}\approx R_{h}$.
\end{proof}

\begin{prop}\label{3.4} If $(R,G)$ is invertible, then for every $g\in G$ there exists a linear transformation $T_{g}$ such that $Ker(T_{g})=R_{g}$.
\end{prop}

\begin{proof} Let $g\in G$. Define $T_{g}:R\rightarrow R/R_{g}$ by $T_{g}(x)=x+R_{g}$. Then it is easy to prove that $T_{g}$ is well-defined linear transformation with $Ker(T_{g})=R_{g}$.
\end{proof}

\section{Weakly Crossed Products}

In this section, we study the relations between weakly crossed products and several properties of graded rings, and we introduce several results concerning weakly crossed products.

\begin{defn} Let $R$ be a $G$-graded ring. Then $(R,G)$ is said to be weakly crossed product if $R_{g}$ contains a unit for all $g\in supp(R,G)$.
\end{defn}

Clearly, if $(R,G)$ is a crossed product, then $(R,G)$ is a weakly crossed product. However, the converse is not true in general; if $|G|\geq2$, then the trivial graduation of $R$ ($R_{e}=R$ and $R_{g}=0$ otherwise) is a weakly crossed product but it is not a crossed product since for $g\neq e$, $R_{g}$ does not contain units.

\begin{exa} In Example \ref{3}, $(R,G)$ is a weakly crossed product since $\left(
                                                                          \begin{array}{cc}
                                                                            1 & 0 \\
                                                                            0 & 1 \\
                                                                          \end{array}
                                                                        \right)\in R_{0}$ and $\left(
                                                                                                 \begin{array}{cc}
                                                                                                   0 & 1 \\
                                                                                                   1 & 0 \\
                                                                                                 \end{array}
                                                                                               \right)\in R_{2}$ are units. However, $(R,G)$ is not a crossed product since $G\neq supp(R,G)$.
\end{exa}

\begin{exa} In Example \ref{5}, $(R,G)$ is not a weakly crossed product since $R_{1}$ does not contain units.
\end{exa}

\begin{prop} If $(R,G)$ is a weakly crossed product, then $(R,G)$ is first strong.\end{prop}

\begin{proof}Let $g\in $supp$(R,G)$. Since $(R,G)$ is weakly crossed product, $R_{g}$ contains a unit, say $r$ and then $1=rr^{-1}\in
R_{g}R_{g^{-1}}$ by Lemma \ref{2.4}. Hence, $(R,G)$ is first strong.\end{proof}

\begin{cor} If $(R,G)$ is a weakly crossed product, then $supp(R,G)$ is a subgroup of $G$.\end{cor}

Also, since every first strongly graded ring is second strongly graded, we can state the following.

\begin{cor} If $(R,G)$ is a weakly crossed product, then $(R,G)$ is second strong.\end{cor}

As every first strongly graded ring is non-degenerate, we have the following.

\begin{cor}\label{18} If $(R,G)$ is a weakly crossed product, then $(R,G)$ is non-degenerate.\end{cor}

\begin{exa}\label{1} Let $R=M_{2}(K)$ (the ring of all $2\times2$ matrices with entries from a field $K$) and $G=\mathbb{Z}$ (the group of integers). Then $R$ is $G$-graded by

 $R_{0}=\left(
          \begin{array}{cc}
            K & 0 \\
            0 & K \\
          \end{array}
        \right)$, $R_{1}=\left(
                           \begin{array}{cc}
                             0 & K \\
                             0 & 0 \\
                           \end{array}
                         \right)$, $R_{-1}=\left(
                                             \begin{array}{cc}
                                               0 & 0 \\
                                               K & 0 \\
                                             \end{array}
                                           \right)$ and $R_{j}=\{0\}$ otherwise. $(R,G)$ is non-degenerate but it is not weakly crossed product since $R_{1}$ does not contain units. So, the converse of Corollary \ref{18} is not true in general.
\end{exa}

\begin{prop}\label{21} If $(R, G)$ is a weakly crossed product, then $(R, G)$ is weak.
\end{prop}

\begin{proof} Let $g\in G$ such that $R_{g}=\{0\}$. If $R_{g^{-1}}\neq\{0\}$, then $g^{-1}\in supp(R, G)$, and then $R_{g^{-1}}$ contains a unit, say $r$. By Lemma \ref{2.4}, $0\neq r^{-1}\in R_{g}$ which is a contradiction. So, $R_{g^{-1}}=\{0\}$, and hence $(R, G)$ is weak.
\end{proof}

\begin{exa}\label{2} Let $R=M_{3}(K)$ and $G=\mathbb{Z}_{2}$ (the group of integers modulo $2$). Then $R$ is $G$-graded by

 $R_{0}=\left(
          \begin{array}{ccc}
            K & 0 & K \\
            0 & K & 0 \\
            K & 0 & K \\
          \end{array}
        \right)$ and $R_{1}=\left(
                              \begin{array}{ccc}
                                0 & K & 0 \\
                                K & 0 & K \\
                                0 & K & 0 \\
                              \end{array}
                            \right)$. $(R, G)$ is weak but not weakly crossed product since $R_{1}$ does not contain units. So, the converse of Proposition \ref{21} is not true in general.
\end{exa}

\begin{prop}\label{19} Let $R$ be a weakly $G$-graded domain. If $(R, G)$ is invertible, then $(R,G)$ is a weakly crossed product.\end{prop}

\begin{proof} Apply Lemma \ref{2.5}.
\end{proof}

\begin{rem} In Example \ref{2.2}, $(R, G)$ is weakly crossed product by Proposition \ref{19}.
\end{rem}

\begin{rem} The converse of Proposition \ref{19} is not true in general; if $R$ is not a field, then the trivial graduation of $R$  is a weakly crossed product but not invertible.
\end{rem}

\begin{prop} \label{20} If $(R,G)$ is a weakly crossed product, then for all $g\in supp(R,G)$, there exists a unit $r\in R_{g}$ such that $R_{g}=R_{e}r$ (i.e., $R_{g}$ is a cyclic $R_{e}$-module).\end{prop}

\begin{proof}Let $g\in supp(R,G)$. Since $(R,G)$ is a weakly crossed product, $R_{g}$ contains a unit, say $r$. Let $x\in R_{g}$. Then
$x=x.1=x.r^{-1}r\in R_{g}R_{g^{-1}}r\subseteq R_{e}r$. On the other hand, $R_{e}r\subseteq R_{e}R_{g}\subseteq R_{g}$. Hence $R_{g}=R_{e}r$.\end{proof}

\begin{rem} Similarly, one can prove that if $(R,G)$ is a weakly crossed product, then for all $g\in supp(R,G)$, there exists a unit $r\in R_{g}$ such that $R_{g}=rR_{e}$.
\end{rem}

\begin{cor} If $(R,G)$ is a weakly crossed product, then $R_{g}$ is isomorphic to $R_{e}$ as an $R_{e}$ - module for all $g\in supp(R,G)$.\end{cor}

\begin{proof} Let $g\in suup(R, G)$. Then by Proposition \ref{20}, $R_{g}=R_{e}r$ for some unit $r\in R_{g}$, and then $f:R_{e}\rightarrow R_{g}$ such that $f(x)=xr$ is an $R_{e}$-isomorphism.
\end{proof}

\begin{prop}Let $R$ be a commutative $G$ - graded ring and $M$ be a $G$-graded $R$-module. If $(R,G)$ is a weakly crossed product, then $M_{g}$ is isomorphic to $M_{e}$ as an $R_{e}$ - module for all $g\in $supp$(R,G)$.\end{prop}

\begin{proof} Let $g\in supp(R,G)$. Since $(R,G)$ is a weakly crossed product, we have $R_{g}=R_{e}r$ for some unit $r\in R_{g}$. Since $(R,G)$ is
first strong, $(M,G)$ is first strong. So, $M_{g}=R_{g}M_{e}=R_{e}rM_{e}=rR_{e}M_{e}=rM_{e}$, and then $f:M_{e}\rightarrow M_{g}$ such that $f(x)=rx$ is an $R_{e}$-isomorphism.\end{proof}

\begin{prop}\label{22} Let $(R,G)$ be a weakly crossed product, $M$ a $G$-graded $R$-module and $f:M\rightarrow R$ an $R$-homomorphism. Then $f(M_{e})=R_{e}$ if and only if $f(M_{g})=R_{g}$ for all $g\in supp(R,G)$.\end{prop}

\begin{proof} Suppose that $f(M_{e})=R_{e}$. Let $g\in supp(R,G)$. Since $(R,G)$ is a weakly crossed product, $R_{g}=rR_{e}$ for some unit $r\in R_{g}$ and $(R,G)$ is first strong, and hence $(M,G)$ is first strong. So, $f(M_{g})=f(R_{g}M_{e})=f(rR_{e}M_{e})=f(rM_{e})=rf(M_{e})=rR_{e}=R_{g}$.
The converse is obvious.\end{proof}

Graded fixing maps have been introduced and studied in \cite{Refai Dawwas}; an $R$-homomorphism $f:M\rightarrow R$ is said to be a grade fixing map if $f(M_{g})\subseteq R_{g}$ for all $g\in G$. It has been proved that if $f$ is a surjective grade fixing map, then $f(M_{g})=R_{g}$ for all $g\in G$. Also, if $f$ is an injective grade fixing map, then $supp(M,G)\subseteq supp(R,G)$.

\begin{prop} Let $(R,G)$ be a weakly crossed product, $M$ a $G$-graded $R$-module and $f:M\rightarrow R$ an $R$-isomorphism. Then $f$ is a grade fixing map if and only if $f(M_{e})=R_{e}$.
\end{prop}

\begin{proof} Suppose that $f$ is a grade fixing map. Since $f$ is surjective, $f(M_{e})=R_{e}$. Conversely, by Proposition \ref{22}, $f(M_{g})=R_{g}$ for all $g\in supp(R,G)$. Since $f$ is injective, $supp(M,G)\subseteq supp(R,G)$. If $g\notin supp(R,G)$, then $g\notin supp(M,G)$ and then $f(M_{g})=f(\{0\})=\{0\}=R_{g}$. Hence, $f$ is a grade fixing map.
\end{proof}

Graded Noetherian modules have been introduced and studied in \cite{Nastasescue}; a graded $R$-module $M$ is said to be graded Noetherian if every descending chain of graded $R$-submodules of $M$ terminates. Clearly, every Noetherian Module is graded Noetherian. However, the converse is not true in general (see \cite{Nastasescue})

\begin{prop} Let $R$ be a $G$-graded domain such that $(R,G)$ is a weakly crossed product, $(R, G)$ is invertible and $supp(R,G)$ is finite. Let $M$ be a $G$-graded $R$-module. Suppose that there exists a bijective grade fixing map from $M$ to $R$. If $M_{e}$ is $R_{e}$-simple, then $M$ is graded Noetherian.
\end{prop}

\begin{proof} Let $f$ be a bijective grade fixing map from $M$ to $R$. Firstly, we prove that $R_{e}$ is $R_{e}$-simple. Let $I$ be a nonzero ideal of $R_{e}$. Then $f^{-1}(I)$ is a nonzero $R_{e}$-submodule of $M_{e}$ and $M_{e}$ is $R_{e}$-simple, $f^{-1}(I)=M_{e}$ and then $I=f(M_{e})=R_{e}$. By Proposition \ref{21} and Corollary \ref{2.7}, $R$ is Noetherian and hence $R$ is graded Noetherian. Let $N_{1}\subseteq N_{2}\subseteq.....$ be graded $R$-submodules of $M$. Then $f(N_{1})\subseteq f(N_{2})\subseteq.....$ are graded ideals of $R$ and since $R$ is graded Noetherian, there exists $m\in \mathbf{N}$ such that $f(N_{n})=f(N_{m})$ for all $n\geq m$ and then $N_{n}=N_{m}$ for all $n\geq m$. Hence, $M$ is graded Noetherian.
\end{proof}

\section{Graded Semi-essential Submodules}

In this section, we introduce and study the concepts of graded semi-essential submodules and graded semi-uniform modules.

\begin{defn} Let $M$ be a graded $R$-module and $K$ a nonzero graded $R$-submodule of $M$. Then $K$ is said to be a graded semi-essential $R$-submodule of $M$ if $K\bigcap P\neq\{0\}$ for every nonzero graded prime $R$-submodule $P$ of $M$. A nonzero graded ideal $I$ of a graded ring $R$ is said to be graded semi-essential if $I\bigcap B\neq\{0\}$ for every nonzero graded prime ideal $B$ of $R$.
\end{defn}

Clearly, every graded essential $R$-submodule is graded semi-essential. The converse is not true in general as we see in the following example.

\begin{exa}\label{Example 1.2e} Consider $R=\mathbb{Z}$ (the ring of integers), $G=\mathbb{Z}_{2}$ (the group of integers modulo $2$) and $M=\mathbb{Z}_{12}[i]=\left\{a+ib:a, b\in \mathbb{Z}_{12}, i^{2}=-1\right\}$ where $\mathbb{Z}_{12}$ is the ring of integers modulo $12$. Then $R$ is $G$-graded by $R_{0}=\mathbb{Z}$ and $R_{1}=\{0\}$, $M$ is $G$-graded by $M_{0}=\mathbb{Z}_{12}$ and $M_{1}=i\mathbb{Z}_{12}$. The graded $R$-submodule $N=\langle6\rangle$ of $M$ is graded semi-essential but not graded essential.
\end{exa}

The following proposition gives a necessary and sufficient condition for a graded $R$-submodule to be a graded semi-essential. The proof is clear and hence it is omitted.

\begin{prop}\label{Theorem 1.3e} Let $M$ be a graded $R$-module and $K$ a nonzero graded $R$-submodule of $M$. Then $K$ is a graded semi-essential $R$-submodule of $M$ if and only if for every graded prime $R$-submodule $P$ of $M$ there exist $r\in h(R)$ and $m\in h(M)$ such that $m\in P$ and $0\neq rm\in K$.
\end{prop}

Also, the proof of the following proposition is straightforward and hence it is omitted.

\begin{prop}\label{Theorem 1.4e} Let $M$ be a graded $R$-module and $K_{1}$, $K_{2}$ two graded $R$-submodules of $M$ such that $K_{1}$ is a graded $R$-submodule of $K_{2}$. If $K_{1}$ is a graded semi-essential $R$-submodule of $M$, then $K_{2}$ is a graded semi-essential $R$-submodule of $M$.
\end{prop}

The next example shows that the converse of Proposition \ref{Theorem 1.4e} is not true in general.

\begin{exa}\label{Example 1.5e} Consider Example \ref{Example 1.2e}. Clearly, $K_{1}=\langle4\rangle$ and $K_{2}=\langle2\rangle$ are graded $R$-submodules of $M$ such that $K_{1}$ is a graded $R$-submodule of $K_{2}$. Since $P=\langle3\rangle$ is a graded prime $R$-submodule of $M$ satisfies $K_{1}\bigcap P=\{0\}$, $K_{1}$ is not a graded semi-essential $R$-submodule of $M$, but it is obvious that $K_{2}$ is a graded semi-essential $R$-submodule of $M$.
\end{exa}

As a result of Proposition \ref{Theorem 1.4e}, we can state the following corollary.

\begin{cor}\label{Corollary 1.6e} Let $M$ be a graded $R$-module and $K_{1}$, $K_{2}$ two graded $R$-submodules of $M$. If $K_{1}\bigcap K_{2}$ is a graded semi-essential $R$-submodule of $M$, then $K_{1}$ and $K_{2}$ are graded semi-essential $R$-submodules of $M$.
\end{cor}

The next example shows that the converse of Corollary \ref{Corollary 1.6e} is not true in general.

\begin{exa}\label{Example 1.7e} Consider $R=\mathbb{Z}$, $G=\mathbb{Z}_{2}$ and $M=\mathbb{Z}_{36}[i]$. Then $R$ is $G$-graded by $R_{0}=\mathbb{Z}$ and $R_{1}=\{0\}$, $M$ is $G$-graded by $M_{0}=\mathbb{Z}_{36}$ and $M_{1}=i\mathbb{Z}_{36}$. Clearly, $K_{1}=\langle12\rangle$ and $K_{2}=\langle18\rangle$ are graded $R$-submodules of $M$. The only graded prime $R$-submodules of $M$ are $P_{1}=\langle2\rangle$ and $P_{2}=\langle3\rangle$. Since $K_{1}\bigcap P_{1}\neq\{0\}$ and $K_{1}\bigcap P_{2}\neq\{0\}$, $K_{1}$ is a graded semi-essential $R$-submodule of $M$. Similarly, $K_{2}$ is a graded semi-essential $R$-submodule of $M$. But $K_{1}\bigcap K_{2}=\{0\}$ which is not a graded semi-essential $R$-submodule of $M$.
\end{exa}

The next Proposition gives a condition under which the converse of Corollary \ref{Corollary 1.6e} is true.

\begin{prop}\label{Theorem 1.8e} Let $M$ be a graded $R$-module and $K_{1}$, $K_{2}$ two graded $R$-submodules of $M$. If $K_{1}$ is a graded essential $R$-submodule of $M$ and $K_{2}$ is a graded semi-essential $R$-submodule of $M$, then $K_{1}\bigcap K_{2}$ is a graded semi-essential $R$-submodule of $M$.
\end{prop}

\begin{proof} It is straightforward.
\end{proof}

Before we give another condition under which the converse of Corollary \ref{Corollary 1.6e} is true, we need the following lemma.

\begin{lem}\label{Lemma 1.9e} Let $M$ be a graded $R$-module, $K$ a graded $R$-submodule of $M$ and $P$ a graded prime $R$-submodule of $M$. If $(K\bigcap P:_{R}m)=Ann(M)$ for all $m\in h(M)-(K\bigcap P)$, then $K\bigcap P$ is a graded prime $R$-submodule of $M$.
\end{lem}

\begin{proof} Let $r\in h(R)$ and $m\in h(M)$ such that $rm\in K\bigcap P$. Suppose that $m\notin K\bigcap P$. Since $rm\in K\bigcap P$, $r\in (K\bigcap P:_{R}m)$ and then $r\in Ann(M)$ which implies that $r\in (K:_{R}M)\bigcap(P:_{R}M)$ and hence $r\in (K\bigcap P:_{R}M)$. Therefore, $K\bigcap P$ is a graded prime $R$-submodule of $M$.
\end{proof}

\begin{prop}\label{Theorem 1.10e}  Let $M$ be a graded $R$-module and $K_{1}$, $K_{2}$ two graded semi-essential $R$-submodules of $M$. If $(K_{1}\bigcap P:_{R}m)=Ann(M)$ for every graded prime $R$-submodule $P$ of $M$ and for every $m\in h(M)-(K_{1}\bigcap P)$, then $K_{1}\bigcap K_{2}$ is a graded semi-essential $R$-submodule of $M$.
\end{prop}

\begin{proof} Let $P$ be a nonzero graded prime $R$-submodule of $M$. Then by Lemma \ref{Lemma 1.9e}, $K_{1}\bigcap P$ is a graded prime $R$-submodule of $M$ and then $(K_{1}\bigcap K_{2})\bigcap P=K_{2}\bigcap(K_{1}\bigcap P)\neq\{0\}$. Hence, $K_{1}\bigcap K_{2}$ is a graded semi-essential $R$-submodule of $M$.
\end{proof}

\begin{prop}\label{Theorem 1.12e} Let $M$ be a graded $R$-module, $K$ a nonzero graded $R$-submodule of $M$ and $T$ a nonzero graded prime $R$-submodule of $M$ such that $(K+T)/T$ is a graded semi-essential $R$-submodule of $M/T$. If $P$ is a graded prime $R$-submodule of $M$ such that $T\subseteq P$ and $K\bigcap P=\{0\}$, then $T=P$.
\end{prop}

\begin{proof} Let $P$ be a graded prime $R$-submodule of $M$ such that $T\subseteq P$ and $K\bigcap P=\{0\}$. Assume that $m\in (K+T)\bigcap P$. Then $m=k+t=p$ for some $k\in K$, $t\in T$ and $p\in P$ and then $k=p-t\in K\bigcap P=\{0\}$ which implies that $k=0$ and hence $m=t\in T$. So, $(K+T)\bigcap P=T$ which implies that $((K+T)/T)\bigcap(P/T))=\{0\}$. But $(K+T)/T$ is a graded semi-essential $R$-submodule of $M/T$ by assumption and $P/T$ is a graded prime $R$-submodule of $M/T$, so $P/T=\{0\}$ which implies that $P=T$.
\end{proof}

\begin{prop}\label{Theorem 1.13e} Let $f:M\rightarrow M^{\prime}$ be a graded $R$-epimorphism such that $Ker(f)\subseteq Grad_{M}(M)$. If $K$ is a graded semi-essential $R$-submodule of $M^{\prime}$, then $f^{-1}(K)$ is a graded semi-essential $R$-submodule of $M$.
\end{prop}

\begin{proof} Let $P$ be a graded prime $R$-submodule of $M$ such that $f^{-1}(K)\bigcap P=\{0\}$. Since $Ker(f)\subseteq Grad_{M}(M)\subseteq P$, $f(P)$ is a graded prime $R$-submodule of $M^{\prime}$ with $K\bigcap f(P)=\{0\}$. Since $K$ is a graded semi-essential $R$-submodule of $M^{\prime}$, $f(P)=\{0\}$ and then $P\subseteq f^{-1}(\{0\})=Ker(f)\subseteq f^{-1}(K)$ and hence $f^{-1}(K)\bigcap P=P$ which implies that $P=\{0\}$. Therefore, $f^{-1}(K)$ is a graded semi-essential $R$-submodule of $M$.
\end{proof}

\begin{prop}\label{Theorem 1.13(1)e} Let $f:M\rightarrow M^{\prime}$ be a graded $R$-isomorphism. If $K$ is a graded semi-essential $R$-submodule of $M$, then $f(K)$ is a graded semi-essential $R$-submodule of $M^{\prime}$.
\end{prop}

\begin{proof} Let $P$ be a nonzero graded prime $R$-submodule of $M^{\prime}$. Since $f$ is epimorphism, $f^{-1}(P)$ is a graded prime $R$-submodule of $M$ and then $K\bigcap f^{-1}(P)\neq\{0\}$ and hence $f(K)\bigcap P\neq\{0\}$ since $f$ is monomorphism. Hence, $f(K)$ is a graded semi-essential $R$-submodule of $M^{\prime}$.
\end{proof}

\begin{prop}\label{Theorem 2.1e} Let $M$ be a graded multiplication faithful $R$-module and $K$ a graded $R$-submodule of $M$. If $(K:_{R}M)$ is a graded semi-essential ideal of $R$, then $K$ is a graded semi-essential $R$-submodule of $M$.
\end{prop}

\begin{proof} Let $P$ be a graded prime $R$-submodule of $M$ such that $K\bigcap P=\{0\}$. Since $M$ is graded multiplication, $K=(K:_{R}M)M$ and $P=(P:_{R}M)M$. Now, $\{0\}=K\bigcap P=((K:_{R}M)M\bigcap(P:_{R}M)M)=((K:_{R}M)\bigcap(P:_{R}M))M$ and since $M$ is faithful, $(K:_{R}M)\bigcap(P:_{R}M)=\{0\}$. But $(K:_{R}M)$ is a graded semi-essential ideal of $R$ and $(P:_{R}M)$ is a graded prime ideal of $R$, so $(P:_{R}M)=\{0\}$ and hence $P=\{0\}$. Therefore, $K$ is a graded semi-essential $R$-submodule of $M$.
\end{proof}

\begin{defn} A nonzero graded $R$-module $M$ is said to be graded semi-uniform if every nonzero graded $R$-submodule of $M$ is graded semi-essential. A nonzero graded ring $R$ is said to be graded semi-uniform if every nonzero graded ideal of $R$ is graded semi-essential.
\end{defn}

Clearly, every graded uniform $R$-module is graded semi-uniform. The converse is not true in general as we see in the following example.

\begin{exa}\label{Example 4.2e} Consider Example \ref{Example 1.7e}, $M=\mathbb{Z}_{36}[i]$ is a graded semi-uniform $\mathbb{Z}$-module. But $M=\mathbb{Z}_{36}[i]$ is not a graded uniform $\mathbb{Z}$-module since $\langle18\rangle\bigcap\langle12\rangle=\{0\}$.
\end{exa}

It is known that the graded uniform property is hereditary. The graded semi-uniform property is not hereditary as we see in the following example.

\begin{exa}\label{Example 4.3e} Consider Example \ref{Example 1.7e}, $M=\mathbb{Z}_{36}[i]$ is a graded semi-uniform $\mathbb{Z}$-module and $K=\langle3\rangle$ is a graded $\mathbb{Z}$-submodule of $M$. The only graded prime $\mathbb{Z}$-submodules of $K$ are $P_{1}=\langle6\rangle$ and $P_{2}=\langle9\rangle$. But $N=\langle12\rangle$ is a graded $\mathbb{Z}$-submodule of $K$ with $N\bigcap P_{2}=\{0\}$. So, $K$ is not a graded semi-uniform $\mathbb{Z}$-module.
\end{exa}

\begin{prop}\label{Theorem 4.5e} Let $R$ be a graded semi-uniform ring. Then every graded multiplication faithful $R$-module is graded semi-uniform.
\end{prop}

\begin{proof} Let $M$ be a graded multiplication faithful $R$-module and $K$ a nonzero graded $R$-submodule of $M$. Since $R$ is graded semi-uniform, $(K:_{R}M)$ is a graded semi-essential ideal of $R$ and then by Proposition \ref{Theorem 2.1e}, $K$ is a graded semi-essential $R$-submodule of $M$ and hence $M$ is a graded semi-uniform $R$-module.
\end{proof}

\end{document}